\documentclass[leqno,12pt]{amsart}
\usepackage{amssymb}
\newcommand{\spaceR}{{\mathbb{R}}}
\newcommand{\spaceC}{{\mathbf C}}
\newcommand{\spaceL}{{\mathbf L}}
\newcommand{\spaceAC}{{\mathbf{AC}}}

\newcommand{\norm}[1]{\lVert#1\rVert}

\newcommand{\w}{\widetilde}
\newcommand{\dd}{\displaystyle}

\renewcommand{\le}{\leqslant}
\renewcommand{\ge}{\geqslant}

\newcommand{\tg}{\tan}

\DeclareMathOperator*{\vraisup}{ess\, sup}
\DeclareMathOperator*{\vraiinf}{ess\, inf}

\newtheorem{theorem}{Theorem}
\newtheorem{lemma}{Lemma}

\newtheorem{assertion}{Proposition}

\theoremstyle{definition}

\theoremstyle{remark}

\title[On minimal periods]{On minimal periods of solutions of higher order functional differential equations
}

\thanks{State National Research Polytechnical University of Perm, Perm,
614990,   Komsomolsky~pr.~29, Russia; bravyi@perm.ru}

\author{E.~Bravyi}

\begin{document}

\begin{abstract}
We show that a problem on minimal periods of solutions of Lipschitz functional differential equations is closely related to the unique solvability of the periodic problem for linear functional differential equations. Sharp bounds for minimal periods of non-constant solutions of higher order functional differential equations with different Lipschitz nonlinearities are obtained.
\end{abstract}

\maketitle



\section{Introduction}
\label{Intro}
Consider a problem on periodic solutions of the equation
\begin{equation}\label{n-3}
x^{(n)}(t)=f(x(\tau(t)),\quad t\in\spaceR^1,
\end{equation}
where $x(t)\in\spaceR^m$, $f:\spaceR^m\to\spaceR^m$ is a Lipschitz function, $\tau:\spaceR^1\to\spaceR^1$ is a measurable function.


If $\tau(t)\equiv t$, the sharp lower estimate
\begin{equation}\label{e:ode}
T\ge 2\pi/L^{1/n}
\end{equation}
for periods $T$ of non-constant periodic solutions to
\eqref{n-3}
is obtained
in \cite{bravyiei1} for $n=1$ and \cite{bravyiei2} for $n\geqslant 1$ for Lipschitz $f$ in the Euclidian norm, and in  \cite{bravyiei3} for even $n$ and Lipschitz functions  $f$ satisfying the condition
\begin{equation}\label{e:3}
\max\limits_{i=1,\ldots,m}|f_i(x)-f_i(\tilde x)|\leqslant L \max\limits_{i=1,\ldots,m}|x_i-\tilde x_i|,\quad x,\w x\in\spaceR^m.
\end{equation}
For equations \eqref{n-3} with an arbitrary piece-wise continuous
deviating argument $\tau$ and Lipschitz $f$
under condition \eqref{e:3},
the best constants in the lower estimates for periods $T$ of non-constant periodic solutions
are found by A.~Zevin for $n=1$ \cite{bravyiei4}
\begin{equation*}
    T\ge 4/L, 
\end{equation*}
and for even $n$ \cite{bravyiei3}
\begin{equation*}
    T\ge \alpha(n)/L^{1/n}. 
\end{equation*}
In the latter case, the best constants $\alpha(n)$ are defined implicitly with the help of solutions to some boundary value problem for an ordinary differential equation of $n$-th order.

Here,  for all $n$, we discover a simple representation of the best constants
in the estimate for minimal periods of non-constant periodic solutions of some more general equations than \eqref{n-3} with
Lipschitz nonlinearities. Some properties of the sequence of the best constants will be obtained. It turns out  that the best constants in lower estimates of periods are the Favard constants.

If equation \eqref{n-3} has a $T$-periodic solution $x$ with absolutely continuous derivatives up to the order $n-1$, then the contraction of $x$ on the interval $[0,T]$ is a solution to  the periodic boundary value problem
\begin{equation}\label{n-4}
x^{(n)}(t)=f(x(\,\w \tau(t))),\quad t\in[0,T],\quad
x^{(i)}(0)=x^{(i)}(T),\quad i=0,\ldots,n-1,
\end{equation}
with $\w \tau(t)=\tau(t+k(t)T)$, $t\in[0,T]$, for some integer $k(t)$ such that $t+k(t)T\in[0,T]$.
If boundary value problem \eqref{n-4} does not have non-constant solutions, then \eqref{n-3} does not have $T$-periodic non-constant solutions either.

Therefore, we can
consider the equivalent periodic boundary value problem for a system of $m$ functional differential equations of the $n$-th order
\begin{equation}\label{e:1}
x^{(n)}(t)=(Fx)(t),\quad t\in[0,T],\quad
x^{(i)}(0)=x^{(i)}(T),\quad i=0,\ldots,n-1,
\end{equation}
where $x\in \spaceAC^{n-1}([0,T],{\spaceR}^m)$. We assume that for the operator
$F:{\spaceC}([0,T],{\spaceR}^m)\to{\spaceL}_\infty([0,T],{\spaceR}^m)$ there exists
a positive constant $L\in\spaceR^1$ such that for all functions $x\in{\spaceC}([0,T],{\spaceR}^m)$ the following inequality holds
\begin{equation}\label{e:2}
\begin{array}{l}
\max\limits_{i=1,\ldots,m}
\left(\vraisup\limits_{t\in[0,T]}(Fx)_i(t)-\vraiinf\limits_{t\in[0,T]}(Fx)_i(t)\right)
\le \\L\,
\max\limits_{i=1,\ldots,m}\left(\max\limits_{t\in[0,T]}x_i(t)-\min\limits_{t\in[0,T]}x_i(t)\right).
\end{array}
\end{equation}
Here and further we use the following functional spaces:
$\spaceC([0,T],{\spaceR}^m)$ is the space of continuous functions $x:[0,T]\to\spaceR^m$;
$\spaceAC^{n-1}([0,T],{\spaceR}^m)$ is the space of  functions with absolutely continuous derivatives
up to order $n-1$;
${\spaceL}_\infty([0,T],{\spaceR}^m)$ is the space of measurable essentially bounded functions $z:[0,T]\to\spaceR^m$ with the norm $\|z\|_{{\spaceL}_\infty}=\max\limits_{i=1,\ldots,m}\vraisup\limits_{t\in[0,T]}|z_i(t)|$;
${\spaceL}_1([0,T],{\spaceR}^m)$ is the space of all integrable functions $z:[0,T]\to\spaceR^m$ with the norm $\|z\|_{{\spaceL}_1}=\max\limits_{i=1,\ldots,m}\int_0^T|z_i(t)|\,dt$.

If in \eqref{e:1} $(Fx)(t)=f(x(\tau(t)))$, $t\in[0,T]$, where $\tau:[0,T]\to[0,T]$ is measurable, then
condition \eqref{e:2} implies that the function $f:{\spaceR}^m\to{\spaceR}^m$ is Lipschitz and satisfies \eqref{e:3}.

Our approach is close to the work \cite{Ronto} where the periodic boundary value problem is considered  on the interval and a general way to obtain the lower estimate of the periods of non-constant  solutions is proposed.

Note that there are a number of papers on minimal periods of non-constant solutions for different classes of equations, in particular, \cite{LY} in Hilbert spaces, \cite{TYL} in Banach spaces with delay, \cite{Vi} in Banach spaces, \cite{Bu} in Banach spaces and difference equations, \cite{Med} in Banach spaces and differentiable delays, \cite{Ar} in spaces $\ell_p$ and $\spaceL_p$.

\section{Main results}
\label{Main}

Define rational constants $K_n$, $n=1,2,\ldots$, by the equalities
\begin{equation}\label{e-55}
K_n=\displaystyle\frac{({2^{n+1}-1})|B_{n+1}|} {2^{n-1}(n+1)!}\text{ if $n$ is odd},
            \quad
          K_n= \frac{|E_{n}|} {4^{n}n!}\text{ if $n$ is even},
\end{equation}
where $B_n$ are the Bernoulli numbers, $E_n$ are the Euler numbers (see, for examples, \cite[p.~804]{AS}).

\begin{assertion}\label{p-1}
\begin{enumerate}
\item[1)] $K_n$ are the Favard constants, the best constants in the inequality
\begin{equation*}
  \max\limits_{t\in[0,1]}  |x(t)|\le K_n\vraisup\limits_{t\in[0,1]} |x^{(n)}(t)|
 \end{equation*}
which holds for all functions $x\in\spaceAC^{n-1}([0,1],\spaceR^1)$
such that $x^{(n)}\in\spaceL_\infty([0,1],R^1)$ and
$x^{(i)}(0)=x^{(i)}(1)$, $i=0,\ldots,n-1$, $\int_0^1 x(t)\,dt=0$,

\item[2)]$\dd
K_n (2\pi)^n=
\min\limits_{\xi\in\spaceR} \int_0^{2\pi}|\phi_n(s)-\xi|\,ds=\frac{4}{\pi}\sum_{k=1}^\infty\frac{(-1)^{(n+1)(k+1)}}{(2k-1)^{n+1}},
$
where
$
   \dd \phi_n(t)=\frac{1}{\pi}\sum_{k=1}^\infty k^{-n}\cos\left(kt-\frac{n\pi}{2}\right),
$
\item [3)]
$\dd
  K_{n+1}=\frac{1}{8(n+1)}\sum_{k=0}^n K_{k}K_{n-k},\  n\geqslant 1,\ K_0=1,\ K_1=1/4,
$
\item [4)]
$
\dd\frac{1}{\cos(t/4)}+\tg(t/4)=1+\sum_{n=1}^\infty K_{n}t^n,\quad |t|<2\pi,
$
\item [5)]
$\lim\limits_{n\to\infty} K_n(2\pi)^n=4/\pi,$
\item [6)]
$
K_1=1/4,\ K_2=1/32,\ K_3=1/192,\ K_4=5/6144,\ K_5=1/7680,\
K_6=61/2949120,\ldots
$
\end{enumerate}
\end{assertion}

\begin{proof}
All these assertions are well known. Proofs of 1), 2), 6) one can see in \cite{Favard,Bernshtein,Levin,Levin1,Bravyi}, 3), 4), 5) in, for example, \cite{Bravyi}.
\end{proof}

\begin{theorem} \label{t-1}
        If $F$ satisfies inequality \eqref{e:2} and  periodic problem \eqref{e:1} has a non-constant solution, then
        \begin{equation}\label{e:4}
            T\ge \frac{1}{(L\,K_n)^{1/n}}.
        \end{equation}
\end{theorem}

To prove Theorem \ref{t-1}, we need two lemmas.
\begin{lemma}\label{l-1}
Let $F$ satisfy \eqref{e:2}. If problem \eqref{e:1} has a non-constant solution, there exist a measurable function $\tau:[0,T]\to[0,T]$ and a constant $C$ such that one of non-constant components of the solution satisfies the scalar periodic boundary problem
\begin{equation}\label{e:1y}
\left\{
\begin{array}{l}
y^{(n)}(t)=L\,y(\tau(t))+C,\quad t\in[0,T],\\[3pt]
y^{(i)}(0)=y^{(i)}(T),\quad i=0,\ldots,n-1.
\end{array}
\right.
\end{equation}
\end{lemma}

\begin{proof}
Suppose $y=x_j$ is a non-constant component of the solution $x$ to \eqref{e:1} for which
the right-hand side of \eqref{e:2} takes the maximum. Then the length of the range of $(Fx)_j$ does not exceed the length of the range of $x_j$ multiplied the constant $L$. So, there exist
a measurable function $\tau:[0,T]\to[0,T]$ and a constant $C$ such that
\begin{equation*}
    (Fx)_j(t)=L\, y(\tau(t))+C
\end{equation*}
for almost all $t\in[0,T]$. This proves the Lemma.
\end{proof}

\begin{lemma}\label{l-2}
Let $L>0$. Problem \eqref{e:1y} has a unique solution for each measurable $\tau:[0,T]\to [0,T]$ and each constant  $C\in\spaceR^1$ if
\begin{equation}\label{e:Ll}
    L<\frac{1}{K_n\,T^n}.
\end{equation}
\end{lemma}

\begin{proof}
Problem \eqref{e:1y} has the Fredholm property \cite{AMR}. Hence, this problem is uniquely solvable if and only if the homogeneous problem
\begin{equation}\label{e:1yh}
y^{(n)}(t)=L\,y(\tau(t)),\  t\in[0,T],
\quad y^{(i)}(0)=y^{(i)}(T),\  i=0,\ldots,n-1.
\end{equation}
has only the trivial solution. Let $y$ be a nontrivial solution of \eqref{e:1yh}.
From \cite{Levin, Levin1} it follows that for some constant $C_1$ and any constant $\xi$  the solution $y$ satisfies the equality
\begin{equation}\label{e:333}
\begin{split}
       y(t)=\frac{T^{n-1}}{(2\pi)^{n-1}}\int_0^T (\phi_n({2\pi s}/{T})-\xi)y^{(n)}(t-s)\,ds+C_1=\\
       \frac{T^{n-1}}{(2\pi)^{n-1}}\int_0^T (\phi_n({2\pi s}/{T})-\xi) L y(\tau(t-s))\,ds+C_1,
\end{split}
\end{equation}
where $t\in[0,T]$, $y(\zeta-T)=y(\zeta)$, $\tau(\zeta-T)=\tau(\zeta)$, $\zeta\in[0,T]$; $\phi_n$ is defined in Proposition \ref{p-1}.
Therefore, if
\begin{equation}
\begin{split}
L<\frac{(2\pi)^{n-1}}{T^{n-1}
\inf\limits_{\xi\in\spaceR} \int_0^T|\phi_n(2\pi s/T)-\xi|\,ds}=\\
\frac{(2\pi)^{n}}{T^{n}\inf\limits_{\xi\in\spaceR} \int_0^{2\pi}|\phi_n(s)-\xi|\,ds}=\frac{1}{K_nT^n},
\end{split}
\end{equation}
then  the linear operator $A$ in the right-hand side of \eqref{e:333}
\begin{equation*}
\begin{split}
       (Ay)(t)=\frac{T^{n-1}L}{(2\pi)^{n-1}}\int_0^T (\phi_n({2\pi s}/{T})-\xi)  y(\tau(t-s))\,ds+C_1,\quad t\in[0,T],
\end{split}
\end{equation*}
is a contraction in  $\spaceL_\infty([0,T],\spaceR^1)$. In this case, for each $C_1$ equation \eqref{e:333}
has a unique solution which is a constant (we use here the equality $\int_0^T \phi_n(2\pi t/T)\,dt=0$). From \eqref{e:1yh} it follows that this constant is zero. Therefore,
problem \eqref{e:1y} is uniquely solvable.
\end{proof}

\begin{proof}[Proof of Theorem \ref{t-1}] Let \eqref{e:1} have a non-constant solution.
From Lem\-ma \ref{l-1} it follows  that the non-constant component $x_j$ (from the proof of Lemma \ref{l-1}) of the solution $x$ to \eqref{e:1} is a solution to \eqref{e:1y} with some constant $C$ and a measurable function $\tau:[0,T]\to[0,T]$. If \eqref{e:Ll}, it follows from Lemma \ref{l-2} that this solution is unique:
$x_j(t)\equiv -C/L$. Then from \eqref{e:2} it follows  that each component $x_i$ of the non-constant solution $x$ is constant. Therefore,
inequality \eqref{e:Ll} does not hold.
\end{proof}

Now assume that an operator  $F$ in \eqref{e:1} acts into the space of integrable functions $\spaceL_1([0,T],\spaceR^m)$.
\begin{theorem}\label{t-2}
Suppose an operator $F$ acts from the space $\spaceC([0,T],{\spaceR}^m)$ into the space ${{\spaceL}_1([0,T],{\spaceR}^m)}$
and there exist positive functions $p_i\in\spaceL_1([0,T],\spaceR^1)$, $i=1,\ldots,m$, such that
for every $x\in\spaceC([0,T],\spaceR^m)$ the inequality
\begin{equation}\label{e:5}
\begin{array}{l}
\displaystyle\max\limits_{i=1,\ldots,m}\left(\underset{t\in[0,T]}{{\rm vrai\,sup}}\,\frac{(Fx)_i(t)}{p_i(t)}-\underset{t\in[0,T]}{{\rm vrai\,inf}}\,\frac{(Fx)_i(t)}{p_i(t)}\right)\\[3pt]
\leqslant
\max\limits_{i=1,\ldots,m}\left(\max\limits_{t\in[0,T]}x_i(t)-\min\limits_{t\in[0,T]}x_i(t)\right)
\end{array}
\end{equation}
holds.
If periodic problem \eqref{e:1} has a non-constant solution, then  for each $i=1,\ldots,n$
\begin{equation}\label{e:4int}
    \norm{p_i}_{\spaceL_1}\ge 4\ \text{ if $\ n=1$},\quad     \norm{p_i}_{\spaceL_1} > \frac{4}{K_{n-1}T^{n-1}}\ \text{ if $\ n\ge2$}.
\end{equation}
\end{theorem}

To prove Theorem \ref{t-2}, we also need two lemmas.
\begin{lemma}\label{l-3}
Let $F$ satisfy inequality \eqref{e:5}.
If problem \eqref{e:1} has a non-constant solution, there exist a measurable function $\tau:[0,T]\to[0,T]$ and a constant $C$ such that one of non-constant components of the solution satisfies the scalar periodic boundary value problem
\begin{equation}\label{e:1yp}
\left\{
\begin{array}{l}
y^{(n)}(t)=p(t)(y(\tau(t))+C),\quad t\in[0,T],\\[3pt]
y^{(i)}(0)=y^{(i)}(T),\quad i=0,\ldots,n-1.
\end{array}
\right.
\end{equation}
\end{lemma}

\begin{proof}
Suppose $y=x_j$ is a non-constant component of the solution $x$ to \eqref{e:1} for which the right-hand side of \eqref{e:5} takes the maximum. Then the length of the range of $(Fx)_j/p_j$ does not exceed the length of the range of $x_j$. So, there exist
a measurable function $\tau:[0,T]\to[0,T]$ and a constant $C$ such that
\begin{equation*}
    (Fx)_j(t)=p(t)(y(\tau(t))+C)\ \text{ for almost all $t\in[0,T]$},
\end{equation*}
where $p=p_j$. This proves the Lemma.
\end{proof}

\begin{lemma}[\cite{5,7,n3,n4,6,8,Bravyi}]\label{l-4}
Let a positive number ${\mathcal P}$ be given.
Problem \eqref{e:1yp} has a unique solution for each measurable $\tau:[0,T]\to [0,T]$ and each non-negative function $p\in\spaceL_1([0,T],\spaceR^1)$ with norm $\norm{p}_{\spaceL_1}={\mathcal P}$ if and only if
\begin{equation}\label{e:pn}
    {\mathcal P}< 4\ \text{ if $\ n=1$},\quad
    {\mathcal P}\le \frac{4}{K_{n-1}T^{n-1}}\ \text{ if $\ n\ge2$}.
\end{equation}
\end{lemma}

For $n=1$, $n=2$, $n=3$, $n=4$  this Lemma is proved in \cite{5,7,n3,n4}, for arbitrary $n$ in \cite{6, 8, Bravyi}.

\begin{proof}[Proof of Theorem \ref{t-2}]
Let \eqref{e:1} have a non-constant solution.
From Lem\-ma \ref{l-3} it follows  that a non-constant component $x_j$ (from the proof of Lemma \ref{l-3}) of the solution $x$ to \eqref{e:1} is a solution to \eqref{e:1yp} with $p=p_j$, some constant $C$, some measurable function $\tau:[0,T]\to[0,T]$. If \eqref{e:pn}, it follows from Lemma \ref{l-4} that the solution $x_j$ is unique: $x_j(t)\equiv -C$. From \eqref{e:5} it follows  that each component $x_i$ of the non-constant solution $x$ is constant. Therefore, inequality \eqref{e:pn} does not  hold.
\end{proof}

\section{The sharpness of estimates}

The estimates \eqref{e:4} and \eqref{e:4int} in Theorems \ref{t-1} and \ref{t-2} are sharp. The sharpness of \eqref{e:4int} is shown in \cite{Bravyi}. The sharpness of \eqref{e:4}
for even $n$ was shown in \cite{bravyiei3} in other terms.
Now for every $n\ge1$ we obtain functions $\tau:[0,T]\to[0,T]$ such that the periodic boundary value problem
\begin{equation}\label{e:L}
    x^{(n)}(t)=Lx(\tau(t)),\quad t\in[0,T],\quad
x^{(i)}(0)=x^{(i)}(T),\quad i=0,\ldots,n-1,
\end{equation}
has a non-constant solution provided that \eqref{e:4} is an
identity:$L=\frac{1}{K_nT^n}$. Find a solution to the auxiliary problem
\begin{equation}\label{e:a}
    x^{(n)}(t)=L\,h(t),\quad t\in[0,T],\quad
x^{(i)}(0)=x^{(i)}(T),\quad i=0,\ldots,n-1,
\end{equation}
where $h(t)=1$ for $t\in[0,T/2]$  and $h(t)=-1$ for $t\in(T/2,T]$. Since $\int_0^T h(t)\,dt=0$, this problem has a solution. It is not unique and defined by the equality
\begin{equation*}
    x(t)=C+L\,\int_0^T  G(t,s) h(s)\,ds,\quad t\in[0,T],
\end{equation*}
where $C$ is an arbitrary constant, $G(t,s)$ is the Green function of the problem
\begin{equation*}
\begin{split}
x^{(n)}(t)=f(t),\quad t\in[0,T],\quad x(0)=0,\ x(T)=0\ (\text{if $n>1$}),\\
x^{(i)}(0)=x^{(i)}(T),\quad i=1,\ldots,n-2\ (\text{if $n>2$}).
\end{split}
\end{equation*}
We have a simple representation for the Green function $G(t,s)$:
\newcommand{\Beta}{{\mathcal B}}
\begin{equation*}
\begin{split}
    G(t,s)=\frac{T^n}{n!}(B_n(t/T)-B_n(0)-\Beta_n((t-s)/T)+B_n(1-s/T)),\\
    \quad t,s\in[0,T],
\end{split}
\end{equation*}
where $B_n(t)$, $n\ge1$, are the Bernoulli polynomials \cite[p.~804]{AS} which can be defined as unique solutions to the problems
\begin{equation*}
\begin{split}
B_n^{(n)}(t)=n!,\quad t\in[0,T],\quad
\int_0^1 B_n(t)\,dt=0,\ B_n^{(i)}(0)=B_n^{(i)}(T),\\
i=0,\ldots,n-2\ (\text{if $n>1$}),
\end{split}
\end{equation*}
$\Beta_n(t)=B_n(\{t\})$ are the periodic Bernoulli functions, $\{t\}$ is the fractional part of $t$.

Using the equality \cite[p.~805, 23.1.11]{AS}
\begin{equation*}
    \int_{t_1}^{t_2}B_n(s)\,ds=(B_{n+1}(t_2)-B_{n+1}(t_1))/(n+1),\quad n\ge1,
\end{equation*}
which is also valid for the functions $\Beta_n(t)$,
we obtain the representation for
solutions $y$ to problem \eqref{e:L}
\begin{equation*}
\begin{split}
    y(t)=C+\frac{2LT^n}{(n+1)!}(B_{n+1}(1/2)-B_{n+1}(0)+B_{n+1}(t/T)-\\
    \Beta_{n+1}(t/T-1/2)), \quad t\in[0,T],\quad C\in\spaceR^1.
\end{split}
\end{equation*}

For even $n=2m$, using \cite[p.~805, 23.19--22, 23.1.15]{AS}
\begin{equation*}
\begin{split}
    B_{2m+1}(1/4)=-B_{2m+1}(3/4)=(2m+1)4^{-2m-1}E_{2m},\\
    \quad B_{2m+1}(1/2)=B_{2m+1}(0)=0,\quad (-1)^mE_{2m}>0,
\end{split}
\end{equation*}
for $C=0$ we obtain that $y(T/4)=-y(3T/4)=(-1)^m$. Therefore, for $C=0$ the function $y$ is a non-constant solution to problem \eqref{e:L},
where $\tau(t)=\left\{\begin{array}{l} T/4\ \text{ if }\ t\in[0,T/2],\\
                                  3T/4\ \text{ if }\ t\in(T/2,T],
                     \end{array}
              \right.$          for $n=0\ {\rm mod}\ 4$,
and $\tau(t)=\left\{\begin{array}{l} 3T/4\ \text{ if }\ t\in[0,T/2],\\
                                  T/4\ \text{ if }\ t\in(T/2,T],
                     \end{array}
              \right.$          for $n=2\ {\rm mod}\ 4$.
Note that these functions $\tau$ were found in \cite{bravyiei3}.

For odd $n=2m-1$ using \cite[p.~805, 23.1.20--21, 23.1.15]{AS}
\begin{equation*}
\begin{split}
B_{2m}=B_{2m}(0)=B_{2m}(1),\ B_{{2m}}(1/2)=(2^{1-{2m}}-1)B_{2m},\\ (-1)^{m+1}B_{2m}>0
\end{split}
\end{equation*}
we have that  $y(0)=-y(T/2)=(-1)^m$ for $C=(-1)^{m}$.
Therefore, for $C=(-1)^{m}$ the function $y$ is a non-constant solution to problem \eqref{e:L}, where $\tau(t)=\left\{\begin{array}{l} T/2\ \text{ if }\ t\in[0,T/2],\\
                                  0\ \text{ if }\ t\in(T/2,T],
                     \end{array}
              \right.$          for $n=1\ {\rm mod}\ 4$,
$\tau(t)=\left\{
\begin{array}{l} 0\ \text{ if }\ t\in[0,T/2],\\
                T/2\ \text{ if }\ t\in(T/2,T],
\end{array}
              \right.$          for $n=3\ {\rm mod}\ 4$.

\section{Example. Equations with ''maxima''}

Let $L$ be a constant,
$\tau, \theta:\spaceR\to\spaceR$ measurable functions such that $\tau(t)\le \theta(t)$ for all $t\in\spaceR$. From Theorem \ref{t-1}, it follows that periods $T$ of non-constants solutions of the equation
\begin{equation*}
    x^{(n)}(t)=L\max\limits_{s\in[\tau(t),\theta(t)]} x(s),\quad t\in\spaceR,
\end{equation*}
satisfy the inequality
        \begin{equation}\label{e:4-2}
           |L|\, T^n\ge \frac{1}{K_n},
        \end{equation}
where the constants $K_n$ are defined by \eqref{e-55}.

Suppose $p:\spaceR\to\spaceR$ is a positive locally integrable $T$-periodic function: $p(t+T)=p(t)$, $p(t)>0$ for all $t\in\spaceR$.
From Theorem \ref{t-2}, it follows that if there exists a $T$-periodic non-constants solution of the equation
\begin{equation*}
    x^{(n)}(t)=p(t)\max\limits_{s\in[\tau(t),\theta(t)]} x(s),\quad t\in\spaceR,
\end{equation*}
then
\begin{equation}\label{e-777}
    \int_0^T p(t)\,dt\ge4\ \text{\ for $n=1$},\quad \int_0^T p(t)\,dt\, T^{n-1}> \frac{4}{K_{n-1}}\ \text{\ for $n\ge2$}.
\end{equation}
Inequalities \eqref{e:4-2} and \eqref{e-777} are sharp.

\section{Conclusion}

Now we formulate unimprovable necessary conditions for the existence of a non-constant periodic solution to \eqref{e:1} which follow from Theorems \ref{t-1} and \ref{t-2}:
if $F$ satisfies \eqref{e:2} and there exists a non-constant solution to \eqref{e:1}, then
$L=L_n$ satisfies the inequalities
\begin{equation*}
\begin{split}
    L_1\ge 4/T,\quad L_2\ge 32/T^2,\quad L_3\ge 132/T^3,\\
    \quad L_4\ge 6144/(5T^4),\quad     L_5\ge 7680/T^5,\ldots;
\end{split}
\end{equation*}
if $F$ satisfies \eqref{e:5} and there exists a non-constant solution to \eqref{e:1},
then  ${\mathcal P}={\mathcal P}_n=\max_{i=1,\ldots,n}\norm{p_i}_{\spaceL_1}$
satisfies the inequalities
\begin{equation*}
\begin{split}
    {\mathcal P}_1\ge 4,\quad {\mathcal P}_2> 16/T,\quad {\mathcal P}_3> 128/T^2,\quad {\mathcal P}_4> 768/T^3,\\
    \quad
    {\mathcal P}_5>  24776/(5T^4),\ldots.
\end{split}
\end{equation*}

It follows from Proposition 1 that $\lim\limits_{n\to\infty} (K_n)^{1/n}=1/(2\pi)$, therefore estimate \eqref{e:4} for large $n$ is close to estimate \eqref{e:ode} for equations without deviating arguments.

New results on existence and uniqueness of periodic solutions for higher order functional differential equations are obtained in  \cite{N5, N4, N3, N2}. Note that Theorems \ref{t-1} and \ref{t-2} cannot be derived from these articles.

\end{document}